\documentclass[11pt,a4paper]{article}
\usepackage[T1]{fontenc}
\usepackage[utf8]{inputenc}
\usepackage{amsmath,amsthm}
\usepackage{newtxtext,newtxmath}
\usepackage[margin=27mm,headheight=15pt]{geometry}
\usepackage{microtype,mathtools,bm}
\usepackage{booktabs,array,enumitem}
\usepackage{xcolor}
\usepackage{fancyhdr,tocloft}
\usepackage[most]{tcolorbox}
\usepackage{hyperref}
\hypersetup{colorlinks=true,linkcolor=black,citecolor=black,urlcolor=black,
 pdftitle={Majorizing-measure bounds in the realized square-function metric},
 pdfauthor={Expository working draft based on force200},
 pdfsubject={Martingale comparison, random metrics, entropy, stochastic convolutions}}
\usepackage{bookmark}
\definecolor{ink}{RGB}{33,48,64}
\definecolor{pale}{RGB}{246,247,249}
\numberwithin{equation}{section}
\newtheorem{theorem}{Theorem}[section]
\newtheorem{lemma}[theorem]{Lemma}
\newtheorem{proposition}[theorem]{Proposition}
\newtheorem{corollary}[theorem]{Corollary}
\theoremstyle{definition}
\newtheorem{definition}[theorem]{Definition}
\newtheorem{example}[theorem]{Example}
\theoremstyle{remark}
\newtheorem{remark}[theorem]{Remark}
\newcommand{\E}{\mathbb E}
\newcommand{\Pp}{\mathbb P}
\newcommand{\R}{\mathbb R}

\newcommand{\F}{\mathcal F}
\newcommand{\G}{\Gamma}
\newcommand{\J}{\mathcal J}
\newcommand{\ind}{\mathbf 1}
\newcommand{\osc}{\operatorname{osc}}

\newcommand{\norm}[1]{\left\|#1\right\|}

\newtcolorbox{orientation}{colback=pale,colframe=ink!30,boxrule=.35pt,
 arc=0pt,left=9pt,right=9pt,top=7pt,bottom=7pt,breakable}
\newcommand{\roadmap}[1]{\begin{orientation}#1\end{orientation}}

\title{\textbf{Majorizing-measure bounds \\ in the realized square-function metric}
\\
\large From martingale comparison to random geometry} 
\author{\normalsize Expository working draft\\[.3em]
\small by Witold M. Bednorz and Rafa{\l} M. {\L}ochowski}
\date{}

\begin{document}
\maketitle
\thispagestyle{empty}
\begin{abstract}
The predictable square function measures the size of a martingale. Applied to
parameter differences, it also defines a random pseudometric. We explain how
to perform a majorizing-measure argument directly in this realized geometry,
without conditioning the terminal field to be Gaussian and without first
replacing the metric by a random multiple of a deterministic one. For a finite
family of predictable Gaussian sums in a $(2,D)$-smooth Banach space, and a fixed
probability measure $\mu$ on the parameter set, we prove
\[
 \norm{\max_j\osc_T f_j}_{L^p}
 \le CD\norm{\G_\mu(d)+\sqrt p\,\Delta_d}_{L^p},\qquad p\ge1,
\]
where $\Delta_d$ is the diameter of the parameter space in the terminal square-function metric and $\G_\mu(d)$ is its
ball-mass integral. The proof combines a localized exponential inequality,
pathwise averaging over random balls, and a stopping-time argument. Atomic
measures recover logarithmically weighted maximal inequalities; Haar measure
connects the result with homogeneous entropy and continuity estimates. We also
give extensions, examples, and a precise account of the restrictions on choosing
the averaging measure.
\end{abstract}

\tableofcontents
\medskip
\section*{Reading guide}
For the main argument, read Sections~\ref{sec:setting}--\ref{sec:goodlambda}
in order. Sections~\ref{sec:consequences} and~\ref{sec:haar} explain the choice
of measure through applications and examples. Section~\ref{sec:contractions}
returns to the original stochastic-convolution problem. The appendix records
the limitation on random measure
selection.

\medskip
\noindent\textbf{Notation at a glance}
\smallskip
\begin{center}
\small
\renewcommand{\arraystretch}{1.22}
\begin{tabular}{@{}p{.19\textwidth}p{.74\textwidth}@{}}
\toprule
Symbol & Meaning \\
\midrule
$j$; $t\in T$ & Filtration time; the separate parameter index. \\
$d_j(t,u)$ & Square function of the parameter difference up to step $j$. \\
$d=d_N$; $\Delta_d$ & Realized terminal pseudometric; its diameter. \\
$\mu$; $B_d(t,r)$ & Fixed probability measure; a possibly random metric ball. \\
$\G_\mu(d)$ & Integral of the square root of the ball-mass logarithm. \\
$D$; $C$ & Smoothness constant; a numerical constant that may change. \\
\bottomrule
\end{tabular}
\end{center}

%\newpage

\section{Introduction: from a square function to a geometry}

A useful way to study a martingale is to replace its fluctuations by an
accumulated conditional variance. If $(M_j)$ is a scalar martingale, its
predictable square function is
\[
 s_N(M)=\left(\sum_{i=1}^N
       \E\bigl[|M_i-M_{i-1}|^2\mid\F_{i-1}\bigr]\right)^{1/2}.
\]
For an indexed family $M_j(t)$, there is another use for the same construction:
\[
 d(t,u)=s_N\bigl(M(t)-M(u)\bigr).
\]
The quantity $d(t,u)$ measures how much noise distinguishes the two parameters.
It is a pseudometric, and it is usually random.

The main goal of this draft is to transfer questions about a
stochastic convolution to questions about such square functions, and then to
use the latter as moduli of continuity. The present manuscript carries out this
programme in a finite-index setting and explains its consequences. Its central
rule is simple:
\begin{center}
\emph{Keep the averaging measure fixed, but allow its metric balls to be random.}
\end{center}
This does not mean that the indexed process is Gaussian after its square
functions have been observed. In general that assertion is false. Instead, we
use an exponential estimate for a \emph{joint event}: an increment is large
while its square function is small.

\subsection{Historical background}

There are two classical strands behind the argument. The first relates the
sample behavior of a random field to the geometry of its increments. Dudley's
entropy bound for Gaussian processes \cite{Dudley67} estimates a supremum by an
integral of square roots of covering-number logarithms. Majorizing-measure and
generic-chaining theory refined this viewpoint; Talagrand's monograph
\cite{Talagrand14} gives a systematic account. For an upper bound, one can often
replace covering numbers by the masses of metric balls under a suitable
probability measure.

The second strand consists of distribution-function and square-function
inequalities for martingales. Burkholder's good-$\lambda$ method
\cite{Burkholder73} converts localized tail bounds into moment bounds. Pinelis
\cite{Pinelis94} developed sharp distributional inequalities in smooth Banach
spaces. The factor $D$ in this paper comes from that smoothness structure; the
factor $\sqrt p$ comes from a Gaussian-scale exponential estimate. For a
textbook treatment of martingale concentration, including the distinction
between conditional exponential estimates and moment bounds, see
Bercu--Delyon--Rio \cite[Chapter~3]{BDR15}.

The use of quadratic variation as a modulus is also established. Nishiyama's
\emph{Entropy Methods for Martingales} \cite[Section~2.4]{Nishiyama00} introduces
a quadratic modulus relative to a prescribed metric and proves localized
entropy bounds for continuous local martingales. Van der Vaart and van Zanten
\cite[Section~3, Lemma~3.1]{VdVvZ05} formulate a majorizing-measure version of
this approach. These results already implement the principle ``quadratic
modulus, then entropy.'' A distinction in the formulation below is that the
ball-mass functional is evaluated at the \emph{realized terminal metric},
rather than only at a fixed metric controlling it through one random scalar.

The averaging method has a direct predecessor in Bednorz's theorem on
majorizing measures \cite{Bednorz06}. Its proof uses normalized averages over
metric balls and a Sobolev-type inequality. In the present argument the balls
are random, so one must take care not to average an exponential estimate
against an auxiliary measure selected from the same terminal randomness. We
instead integrate the localized estimate against a fixed product measure and
only afterwards use random balls pathwise.

The connection with stochastic convolutions comes from
Kwapie\'n--Marcus--Rosi\'nski \cite{KMR06}. They prove, among other results, that
convolving a semimartingale with an independent continuous Gaussian kernel
having stationary increments and vanishing at zero gives a continuous process.
In their Section~4, translation invariance expresses entropy through the
Lebesgue measure of metric balls. A further step averages a random entropy
expression, using a result of Fernique. Here the same kind of homogeneous
ball-mass functional is applied to martingale square functions; the
probabilistic input does not require an independent Gaussian kernel.

There is particularly close sequential work. Block, Dagan and Rakhlin
\cite{BlockDaganRakhlin21} use a fixed measure and path-dependent balls for
predictable Rademacher trees. Their stated complexity takes a worst case over
noise paths, but intermediate estimates in Appendix~A also retain realized
ball masses. Thus neither random balls nor sequential majorizing measures
should be claimed as new here. The precise comparison concerns the
self-scaled $L^p$ estimate below, including its random diameter and its
Banach-space and time-maximal formulations.

Exponential estimates for convolutions were developed by Seidler
\cite{Seidler10}. Van Neerven and Veraar \cite{NV20,NV22} develop sharp maximal
inequalities for stochastic convolutions, including predictable contraction
recursions and evolution-family applications. Cox and van Winden
\cite{CoxVW24} prove a sharp bound for the supremum of countably many
convolutions, with weights $\sqrt{p+\log k}$, and derive modulus-of-continuity
estimates. Their argument uses exponential tails, a union bound and
good-$\lambda$; it also acknowledges earlier extrapolation work of Geiss
\cite{Geiss97}. We recover the form of that maximal bound by assigning mass of
order $k^{-2}$ to the $k$th index. 

\subsection{The contribution formulated in this draft}

The main statement is a bound of the form
\begin{equation}\label{eq:intro}
 \norm{\max_j\osc_T f_j}_{L^p}
 \le CD\norm{\G_\mu(d_\omega)+\sqrt p\,\Delta_{d_\omega}}_{L^p}.
\end{equation}
Both the shape and the diameter of the metric on the right may depend on the
same innovations that generate $f$. The measure $\mu$ is chosen before those
innovations, or conditionally on initial information. No optimization over an
arbitrary terminal random measure is implicit in \eqref{eq:intro}.

The main proof is organized so that its two operations are easy to distinguish.
First, a localized martingale inequality is integrated over all parameter pairs
with respect to $\mu\otimes\mu$. Second, a deterministic averaging argument is
performed on each realization of the resulting random metric. A stopping-time
argument then removes a deterministic diameter cutoff. Sections~\ref{sec:consequences}
and~\ref{sec:haar} explain why different choices of $\mu$ lead to the familiar
weighted and homogeneous estimates.

The authors acknowledge that they used AI tools to improve the preliminary version of this manuscript.

\section{Setting, notation and the main theorem}\label{sec:setting}

\subsection{Two indices with different roles}

The integer $j$ is the filtration-time index. The element $t\in T$ is the
parameter in which a supremum or a continuity property is studied. They need
not represent the same physical variable.

Let $T$ be a finite nonempty set. Let $X$ be a real separable Banach space with
\begin{equation}\label{eq:smooth}
 \norm{x+y}^2+\norm{x-y}^2
 \le 2\norm{x}^2+2D^2\norm{y}^2,\qquad x,y\in X,
\end{equation}
where $D\ge1$. This is the $(2,D)$-smoothness convention of \cite{NV22}.
Hilbert spaces have $D=1$. The spaces $L^q$, $2\le q<\infty$, have
$D=\sqrt{q-1}$; see \cite[Section~2]{NV22}.

Let $(\F_j)_{j=0}^N$ be a filtration. For every $i$, let $\gamma_i$ be a standard
Gaussian vector in $\R^{m_i}$, independent of $\F_{i-1}$ and measurable with
respect to $\F_i$. Let $A_i(t):\R^{m_i}\to X$ be $\F_{i-1}$-measurable,
with finite operator norm almost surely. Define
\begin{equation}\label{eq:model}
 f_0(t)=0,\qquad f_j(t)=\sum_{i=1}^j A_i(t)\gamma_i.
\end{equation}
For an operator $A:\R^m\to X$, write
\[
 \norm{A}_\gamma
   =\left(\E\norm{Ag}^2\right)^{1/2},\qquad g\sim N(0,I_m).
\]
This is the finite-dimensional Gaussian-radonifying norm. In a Hilbert target,
it is the Hilbert--Schmidt norm.

\begin{definition}[The square-function metric]
For $j\le N$, put
\begin{align}
 d_j(t,u)^2
 &=\sum_{i=1}^j\norm{A_i(t)-A_i(u)}_\gamma^2\label{eq:metric}\\
 &=\sum_{i=1}^j\E\!\left[
    \norm{(A_i(t)-A_i(u))\gamma_i}^2\mid\F_{i-1}\right].\nonumber
\end{align}
We abbreviate $d=d_N$ and $\Delta_d=\max_{t,u\in T}d(t,u)$.
\end{definition}

For each realization, $d_j$ is a pseudometric: the triangle inequality is
Minkowski's inequality in the $\ell^2$ sum of the operator-normed spaces. It may
vanish at distinct parameters. This is appropriate: if $d(t,u)=0$, every
$A_i(t)-A_i(u)$ is zero, so the two martingale paths agree.

\begin{remark}[Which square function?]
The relevant quantity is $s_N(f(t)-f(u))$, not
$|s_N(f(t))-s_N(f(u))|$. For example, the variables $\gamma$ and $-\gamma$
have equal individual square functions, but their difference has square
function $2$. In a contraction recursion, the square function is formed from
the innovations; see Section~\ref{sec:contractions}.
\end{remark}

\subsection{A fixed measure and random balls}

Fix a deterministic probability measure $\mu$ on $T$ such that $\mu(t)>0$ for
every $t$. We use closed balls
\[
 B_d(t,r)=\{u\in T:d(t,u)\le r\}.
\]
The choice of closed versus open balls does not affect the integrals below.

\begin{definition}[The ball-mass functional]
For a pseudometric $d$ on $T$, set
\begin{equation}\label{eq:Gamma}
 \G_\mu(d)=\sup_{t\in T}\int_0^{\Delta_d}
       \sqrt{\log\frac{e}{\mu(B_d(t,r))}}\,dr.
\end{equation}
If $\Delta_d=0$, the integral is zero.
\end{definition}

The integrand is large when the ball has small mass. Its integral records the
cost of refining an average around $t$ from the whole parameter set down to
arbitrarily small neighborhoods. The factor $e$ has the convenient consequence
\begin{equation}\label{eq:Gamma-monotone}
 \Delta_d\le\G_\mu(d),\qquad
 d'\le d\ \Longrightarrow\ 
 \Delta_{d'}\le\Delta_d,\quad \G_\mu(d')\le\G_\mu(d).
\end{equation}
The second assertion follows because smaller distances give larger balls.

\begin{theorem}[Realized square-function comparison]\label{thm:main}
Under the preceding assumptions, there is a numerical constant $C$ such that,
for every $p\ge1$,
\begin{equation}\label{eq:main}
 \norm{\max_{0\le j\le N}\osc_T f_j}_{L^p}
 \le CD\norm{\G_\mu(d)+\sqrt p\,\Delta_d}_{L^p},
\end{equation}
where $\osc_T f_j=\max_{t,u\in T}\norm{f_j(t)-f_j(u)}$.
The constant is independent of $T$, $N$, the dimensions $m_i$, and $p$.
In particular,
\begin{equation}\label{eq:mainL1}
 \E\max_j\osc_T f_j\le CD\E\G_\mu(d).
\end{equation}
\end{theorem}

The inequality is understood in the extended sense when the right-hand side is
infinite. It is an upper bound, not a two-sided characterization. The fixed
measure may be optimized outside the norm:
\[
 \inf_{\mu\ \mathrm{deterministic}}
       \norm{\G_\mu(d)+\sqrt p\,\Delta_d}_{L^p}.
\]
It may not be optimized independently for every terminal outcome without an
additional argument or a selection cost.

\roadmap{\textbf{Proof map.}
Section~\ref{sec:gaussian} establishes a localized tail estimate for each
parameter difference. Section~\ref{sec:averaging} converts all those estimates
into one bound involving the random ball masses. Section~\ref{sec:goodlambda}
uses first crossings to recover an $L^p$ estimate at the random diameter. The
averaging part does not require that the terminal field be Gaussian.}

\section{The probabilistic input: a localized exponential inequality}\label{sec:gaussian}

\subsection{Gaussian subdivision}

A useful consequence of smoothness is the following conditional inequality,
proved in \cite[Lemma~2.4]{NV22}: if $\eta$ is centered and $x$ is fixed, then
\begin{equation}\label{eq:basiccosh}
 \E\cosh(\lambda\norm{x+\eta})
 \le\left[1+D^2\E\big(e^{\lambda\norm{\eta}}-1-
                      \lambda\norm{\eta}\big)\right]
             \cosh(\lambda\norm{x}).
\end{equation}
The cited bounded-variable version extends to a finite-rank Gaussian variable
by symmetric truncation and dominated convergence.

\begin{lemma}[One Gaussian innovation]\label{lem:gaussian}
If $G$ is a centered finite-rank Gaussian vector in $X$, then for $x\in X$ and
$\lambda\ge0$,
\begin{equation}\label{eq:gaussian-step}
 \E\cosh(\lambda\norm{x+G})
 \le\cosh(\lambda\norm{x})
       \exp\left(\frac{D^2\lambda^2}{2}\E\norm{G}^2\right).
\end{equation}
\end{lemma}

\begin{proof}
Let $G_1,\ldots,G_m$ be independent copies of $G$. Gaussianity gives
$G\stackrel{d}=m^{-1/2}\sum_{\ell=1}^mG_\ell$. Apply
\eqref{eq:basiccosh} successively to the $m$ summands. The result is
\[
 \E\cosh(\lambda\norm{x+G})
 \le\cosh(\lambda\norm{x})
 \left[1+D^2\E\left(e^{\lambda\norm{G}/\sqrt m}-1-
                      \frac{\lambda\norm{G}}{\sqrt m}\right)\right]^m.
\]
A finite-rank Gaussian norm has exponential moments of every linear order.
Taylor's formula and dominated convergence therefore give
\[
 \E\left(e^{\lambda\norm{G}/\sqrt m}-1-
          \frac{\lambda\norm{G}}{\sqrt m}\right)
 =\frac{\lambda^2}{2m}\E\norm{G}^2+o(m^{-1}).
\]
Letting $m\to\infty$ proves \eqref{eq:gaussian-step}.
\end{proof}

The point of subdivision is that all terms above second order disappear in the
limit. One does not need to estimate the entire exponential series at once.
The same proof works conditionally when $x$ and the covariance of $G$ are
measurable with respect to the conditioning sigma-field.

\subsection{A supermartingale and its consequence}

Fix $t,u\in T$. Lemma~\ref{lem:gaussian}, applied conditionally, shows that
\begin{equation}\label{eq:compensated}
 Z_j^\lambda(t,u)=
 \cosh\bigl(\lambda\norm{f_j(t)-f_j(u)}\bigr)
 \exp\left(-\frac{D^2\lambda^2}{2}d_j(t,u)^2\right)
\end{equation}
is a nonnegative supermartingale starting at $1$. Localization removes any
initial integrability restrictions.

\begin{lemma}[Joint increment--energy tail]\label{lem:tail}
With $U(t,u)=\max_{j\le N}\norm{f_j(t)-f_j(u)}$, one has
\begin{equation}\label{eq:localtail}
 \Pp\{U(t,u)>a,\ d(t,u)\le r\}
 \le 2\exp\left(-\frac{a^2}{2D^2r^2}\right),
 \qquad a,r>0.
\end{equation}
\end{lemma}

\begin{proof}
At the first crossing of $a$, on the event that $d(t,u)\le r$, the
supermartingale in \eqref{eq:compensated} is at least
\[
 \frac12\exp\left(\lambda a-\frac{D^2\lambda^2r^2}{2}\right).
\]
The maximal inequality for a nonnegative supermartingale therefore bounds the
probability by $2\exp(-\lambda a+D^2\lambda^2r^2/2)$. Choose
$\lambda=a/(D^2r^2)$.
\end{proof}

\begin{remark}[A joint bound is not a conditional Gaussian law]
Equation~\eqref{eq:localtail} estimates the intersection of two events. It does
not say that $f_N(t)-f_N(u)$ is Gaussian conditional on $d$. Observing the
terminal coefficients may reveal information about earlier innovations. This
is why a direct application of Gaussian comparison after freezing all
coefficients is not used.
\end{remark}

\subsection{Filling the discrete increments continuously}\label{subsec:fill}

For the good-$\lambda$ argument, it is convenient to remove Gaussian
overshoots. On an enlarged probability space, take independent Brownian bridges
$\beta_i$ on $[0,1]$, independent of all the original variables, and set
\[
 W_i(s)=s\gamma_i+\beta_i(s),\qquad 0\le s\le1.
\]
Each $W_i$ is a standard $\R^{m_i}$-valued Brownian motion and is independent of
the information available before interval $i$. On the artificial interval
$[i-1,i]$, define
\begin{equation}\label{eq:interpolation}
 \widetilde f_{i-1+s}(t)
  =f_{i-1}(t)+A_i(t)W_i(s),\qquad 0\le s\le1.
\end{equation}
Within interval $i$, reveal only the path of $W_i$ up to the current time,
not the bridge $\beta_i$ or the endpoint $\gamma_i$ separately. Reveal the
original $\F_i$ at the end of the interval. Together with completed past
segments, this defines a filtration for which the field is a continuous
martingale, with exactly the original values at integer times.

Its accumulated energy during interval $i$ is
\[
 d_{i-1}(t,u)^2+s\norm{A_i(t)-A_i(u)}_\gamma^2.
\]
Lemma~\ref{lem:gaussian} applies to each deterministic Brownian subinterval, so
\eqref{eq:compensated} has a continuous-time supermartingale version. Optional
sampling gives the conditional form of \eqref{eq:localtail} after a stopping
time $\tau$. The remaining terminal metric is
\begin{equation}\label{eq:remainingmetric}
 (d^\tau(t,u))^2
   =\sum_{i=1}^N |(i-1,i]\cap(\tau,N]|\,
                   \norm{A_i(t)-A_i(u)}_\gamma^2
   \le d(t,u)^2.
\end{equation}
Here $|\cdot|$ means interval length. The restarted field is
$\widetilde f_s-\widetilde f_\tau$, $s\ge\tau$.

This construction is only a proof device. We do not replace the terminal
indexed family by an independent Gaussian process. In the following two
sections we work with $\widetilde f$ and omit the tilde; its supremum dominates
the original discrete supremum.

\section{Averaging over terminal random balls}\label{sec:averaging}

\subsection{Integrate first, then choose the balls}

For a deterministic radius $r>0$, define
\begin{equation}\label{eq:Zr}
 \mathcal Z_r=\iint
     \ind_{\{d(v,w)\le r\}}
     \exp\left(\frac{U(v,w)^2}{8D^2r^2}\right)
     \,d\mu(v)d\mu(w),
\end{equation}
where now $U(v,w)=\sup_{s\le N}\norm{f_s(v)-f_s(w)}$.

\begin{lemma}[A bounded exponential average]\label{lem:Z}
For every deterministic $r>0$,
\begin{equation}\label{eq:EZ}
 \E\mathcal Z_r\le2.
\end{equation}
The same bound holds conditionally for a field restarted at a stopping time,
with its remaining metric.
\end{lemma}

\begin{proof}
For a fixed pair $(v,w)$, integration of \eqref{eq:localtail} gives
\begin{align*}
 &\E\left[\ind_{\{d(v,w)\le r\}}
       \exp\left(\frac{U(v,w)^2}{8D^2r^2}\right)\right] \le 1+\int_1^\infty 2y^{-4}\,dy
   =\frac53\le2.
\end{align*}
Integrate against the fixed probability measure $\mu\otimes\mu$. The
conditional assertion follows from the conditional localized tail after
$\tau$.
\end{proof}

No independence between the indicator and $U$ is required. The indicator is
precisely the energy restriction already present in the joint tail estimate.
What matters is that the measure of integration is not selected afterwards
from the terminal noise.

\subsection{The pathwise averaging calculation}

Fix a deterministic $R>0$. We work on $\{\Delta_d\le R\}$ and put
\[
 r_n=2^{-n}R,\quad B_n(t)=B_d(t,r_n),\quad
 m_n(t)=\mu(B_n(t)),\qquad n\ge0.
\]
For a realized field, let
\begin{equation}\label{eq:Pn}
 P_n f_s(t)=\frac1{m_n(t)}\int_{B_n(t)}f_s(v)\,d\mu(v).
\end{equation}
The balls are random, and we make no assertion that these averages are
martingales. They are used only in a deterministic inequality on each outcome.

Since $B_0(t)=T$,
\[
 P_0f_s(t)=\bar f_s:=\int_T f_s(v)\,d\mu(v).
\]
As $n\to\infty$, the ball becomes the zero-distance class of $t$. The field is
constant on that class, so $P_nf_s(t)\to f_s(t)$. The finite-index assumption
makes this convergence immediate.

For $v\in B_n(t)$ and $w\in B_{n-1}(t)$, the triangle inequality gives
$d(v,w)\le3r_n$. Expressing $P_nf_s(t)-P_{n-1}f_s(t)$ as the average of
$f_s(v)-f_s(w)$ and using Jensen's inequality yields
\begin{equation}\label{eq:Jensenballs}
 \exp\left(
  \frac{\norm{P_nf_s(t)-P_{n-1}f_s(t)}^2}{8D^2(3r_n)^2}
 \right)
 \le\frac{\mathcal Z_{3r_n}}{m_n(t)m_{n-1}(t)}.
\end{equation}
The product set on the left is only a subset of all pairs with distance at
most $3r_n$; nonnegativity allows its integral to be bounded by
$\mathcal Z_{3r_n}$. This observation replaces the usual selection of a net.

Taking logarithms, enlarging $\log\mathcal Z$ to $\log(1+\mathcal Z)$, and
using $\sqrt{a+b}\le\sqrt a+\sqrt b$, we obtain
\begin{align}
 \norm{P_nf_s(t)-P_{n-1}f_s(t)}
 \le 6\sqrt2 D r_n\bigg[&
       \sqrt{\log(1+\mathcal Z_{3r_n})}\nonumber\\[-.3em]
       &+\sqrt{\log\frac1{m_n(t)m_{n-1}(t)}}\bigg].\label{eq:incrementaverage}
\end{align}
This holds for all centers and all times simultaneously. The cost of a random
ball is its mass under the fixed measure.

\subsection{Summing the geometry and the remaining randomness}

The masses satisfy $m_{n-1}(t)\ge m_n(t)$. Hence the second term in
\eqref{eq:incrementaverage} is bounded by
$\sqrt{2\log(1/m_n(t))}$. The dyadic integral comparison
\begin{equation}\label{eq:dyadicintegral}
 r_n\sqrt{\log\frac1{\mu(B_d(t,r_n))}}
 \le2\int_{r_n/2}^{r_n}
       \sqrt{\log\frac1{\mu(B_d(t,r))}}\,dr
\end{equation}
shows that its sum is controlled by $\G_\mu(d)$. The intervals of integration
are disjoint; above $\Delta_d$ the logarithm without $e$ is zero.

For the first term, define
\begin{equation}\label{eq:VR}
 V_R=\sum_{n\ge1}2^{-n}
             \sqrt{\log(1+\mathcal Z_{3R2^{-n}})}.
\end{equation}
Since the weights sum to one and $x\mapsto e^{x^2}$ is convex for $x\ge0$,
Lemma~\ref{lem:Z} implies
\begin{equation}\label{eq:VRexp}
 \E e^{V_R^2}
 \le\sum_{n\ge1}2^{-n}\E(1+\mathcal Z_{3R2^{-n}})\le3.
\end{equation}
The infinite sum follows by monotone convergence from finite weighted sums
with the missing weight assigned to zero. No independence across scales is
used.

\begin{proposition}[Localized random-geometry bound]\label{prop:localized}
Let $Y^*=\sup_{s\le N}\sup_t\norm{f_s(t)-\bar f_s}$. One may take a numerical
constant $c_0=32$ such that
\begin{equation}\label{eq:pathwise-local}
 Y^*\le c_0D\G_\mu(d)+c_0DR V_R
 \quad\text{on }\{\Delta_d\le R\}.
\end{equation}
Consequently, for $R,u>0$,
\begin{equation}\label{eq:localgeometry}
 \Pp\{Y^*>c_0D\G_\mu(d)+c_0DRu,\ \Delta_d\le R\}
 \le3e^{-u^2}.
\end{equation}
The estimate also holds conditionally for restarted fields.
\end{proposition}

\begin{proof}
Telescope $P_nf_s(t)-P_{n-1}f_s(t)$, use
\eqref{eq:incrementaverage} and \eqref{eq:dyadicintegral}, and take the
supremum over $s,t$. The geometric contribution is at most $24D\G_\mu(d)$,
and the other contribution is at most $6\sqrt2DR V_R$. Both are bounded with
$c_0=32$. Equation~\eqref{eq:VRexp} and Markov's inequality prove the tail
bound. Every step has a conditional counterpart.
\end{proof}

\roadmap{\textbf{What has been gained?}
The metric in \eqref{eq:localgeometry} is the terminal random metric, not a
prescribed comparison metric. We have not conditioned on it. Random balls were
chosen only after the exponential increment estimate had been integrated
against a fixed product measure.}

\section{From a deterministic cutoff to the random diameter}\label{sec:goodlambda}

A tempting but invalid shortcut would be to put $R=\Delta_d$ directly in
\eqref{eq:localgeometry}. The estimate was proved for each deterministic $R$;
a terminal random choice requires additional work. The following stopping-time
argument supplies it.

\begin{proof}[Proof of Theorem~\ref{thm:main}]
First suppose the predictable coefficients are bounded. For the continuous
interpolation, let
\[
 Y_s=\sup_t\norm{f_s(t)-\bar f_s},\qquad
 H_p=D\bigl(\G_\mu(d)+\sqrt p\,\Delta_d\bigr).
\]
The process $Y_s$ is continuous, and $Y^*=\sup_sY_s$ has finite moments.
Fix $\lambda>0$ and let $\tau$ be the first crossing of level $\lambda$,
with $\tau=\infty$ if there is none. On $\{Y^*>\lambda\}$, continuity gives
$Y_\tau=\lambda$.

If $Y^*>2\lambda$, the centered restarted field
\[
 (f_s-f_\tau)-\int_T(f_s(v)-f_\tau(v))\,d\mu(v),\qquad s\ge\tau,
\]
has supremum norm greater than $\lambda$. Its metric $d^\tau$ satisfies
\eqref{eq:remainingmetric}, and therefore
\[
 \G_\mu(d^\tau)\le\G_\mu(d),\qquad
 \Delta_{d^\tau}\le\Delta_d.
\]
On $\{H_p\le\delta\lambda\}$, this gives
\[
 \G_\mu(d^\tau)\le\frac{\delta\lambda}{D},\qquad
 \Delta_{d^\tau}\le R:=\frac{\delta\lambda}{D\sqrt p}.
\]
Choose $\delta=(4c_0)^{-1}$ and apply the conditional version of
\eqref{eq:localgeometry} to the restarted field with $u=3\sqrt p$.
The two terms in its threshold are at most $\lambda/4$ and $3\lambda/4$.
Thus
\begin{equation}\label{eq:goodlambda}
 \Pp\{Y^*>2\lambda,\ H_p\le\delta\lambda\}
 \le \varepsilon_p\Pp\{Y^*>\lambda\},\qquad
 \varepsilon_p=3e^{-9p}.
\end{equation}
In deriving this inequality, the event involving the full future geometry is
bounded by the larger residual-geometry event before conditioning on
$\F_\tau$. It need not itself be $\F_\tau$-measurable.

It follows that
\[
 \Pp\{Y^*>2\lambda\}
 \le\Pp\{H_p>\delta\lambda\}
       +\varepsilon_p\Pp\{Y^*>\lambda\}.
\]
Integrate against $p\lambda^{p-1}\,d\lambda$. Since
$2^p\varepsilon_p<1/2$ for $p\ge1$,
\[
 (1-2^p\varepsilon_p)\E(Y^*)^p
 \le2^p\delta^{-p}\E H_p^p.
\]
Taking roots gives $\norm{Y^*}_{L^p}\le C\norm{H_p}_{L^p}$ with a numerical
$C$. Since $\osc_T f_j\le2Y_j$, this proves \eqref{eq:main} for bounded
coefficients.

For general coefficients, kill the current and all subsequent innovations
at the first index $i$ for which $\max_t\norm{A_i(t)}_\gamma>L$. This is a
predictable operation. The truncated terminal metric $d^{(L)}$ satisfies
$d^{(L)}\le d$, so its right-hand side is bounded by the original one through
\eqref{eq:Gamma-monotone}. There are only finitely many coefficients, and
all are finite almost surely; consequently, for each outcome the truncated
field eventually equals the original field. Fatou's lemma proves the general
inequality. Equation~\eqref{eq:mainL1} follows by taking $p=1$ and using
$\Delta_d\le\G_\mu(d)$.
\end{proof}

\begin{remark}[Where continuity in filtration time was used]
Only the first-crossing step needed continuous interpolation. The ball
averaging was entirely pathwise. This distinction is useful for other noise
models: one needs either continuous crossings or a separate control of the
overshoot. Section~\ref{sec:extensions} gives both types of example.
\end{remark}

\section{Consequences of choosing the measure}\label{sec:consequences}

\subsection{Countable suprema and logarithmic weights}

Suppose $f_j^k$ is a finite collection of martingales of the form
\eqref{eq:model}, driven by a common finite-dimensional Gaussian innovation at
each step. Independent Gaussian noises can be included by taking their direct
sum. Put
\[
 a_k=s_N(f^k)=\left(\sum_i\norm{A_i^k}_\gamma^2\right)^{1/2}.
\]
Add an index $0$ with $f_j^0=0$, and choose on $\{0,1,\ldots,K\}$
\begin{equation}\label{eq:atomic}
 \mu(k)=\frac{1}{2k(k+1)},\quad 1\le k\le K,
 \qquad \mu(0)=1-\sum_{k=1}^K\mu(k)\ge\frac12.
\end{equation}
For the actual square-function metric,
\[
 d(k,0)=a_k,\qquad \Delta_d\le2\max_k a_k.
\]
A ball centered at $k$ has mass at least $\mu(k)$ at every radius, and mass at
least $1/2$ when $r\ge a_k$. Therefore,
\begin{align*}
 \int_0^{\Delta_d}\sqrt{\log\frac{e}{\mu(B_d(k,r))}}\,dr
 &\le a_k\sqrt{\log\frac{e}{\mu(k)}}+
                  \Delta_d\sqrt{\log(2e)}\\
 &\le C a_k\sqrt{\log(k+1)}+C\Delta_d.
\end{align*}
The ball at $0$ is even simpler, since it always has mass at least $1/2$.
We obtain
\[
 \G_\mu(d)+\sqrt p\,\Delta_d
 \le C\sup_{k\le K}\sqrt{p+\log(k+1)}\,a_k.
\]

\begin{corollary}[Weighted maximal inequality]\label{cor:weighted}
For $p\ge1$,
\begin{equation}\label{eq:weighted}
 \norm{\sup_{k\le K}\max_{j\le N}\norm{f_j^k}}_{L^p}
 \le CD\norm{\sup_{k\le K}\sqrt{p+\log(k+1)}\,s_N(f^k)}_{L^p}.
\end{equation}
The same assertion holds for a countable collection by passage through finite
subcollections.
\end{corollary}

The artificial zero index turns a bound for oscillations into a bound for
absolute sizes. The weight $\sqrt{\log k}$ comes from the mass assigned to the
$k$th process. This explains geometrically the index dependence in the bound of
Cox--van Winden \cite{CoxVW24}; the convolution version is discussed in
Section~\ref{sec:contractions}.

\subsection{A choice among several geometries, with an explicit cost}

\begin{corollary}[Mixtures of averaging measures]\label{cor:mixture}
Let $(\mu_\ell)_{\ell\ge1}$ be prescribed full-support probability measures
on $T$, and let $\pi_\ell>0$ with $\sum_\ell\pi_\ell=1$. Then
\begin{equation}\label{eq:mixture}
 \norm{\max_j\osc_T f_j}_{L^p}
 \le CD\left\|
 \inf_\ell\left\{\G_{\mu_\ell}(d)+
       \Delta_d\sqrt{p+\log(1/\pi_\ell)}\right\}
 \right\|_{L^p}.
\end{equation}
\end{corollary}

\begin{proof}
Use the fixed mixture $\mu=\sum_\ell\pi_\ell\mu_\ell$. For every ball,
$\mu(B)\ge\pi_\ell\mu_\ell(B)$. Hence
\[
 \G_\mu(d)\le\G_{\mu_\ell}(d)+
                       \Delta_d\sqrt{\log(1/\pi_\ell)}.
\]
This holds pathwise for every $\ell$, so the infimum may be taken after seeing
$d$. Combine it with Theorem~\ref{thm:main} and
$\sqrt p+\sqrt q\le\sqrt{2(p+q)}$.
\end{proof}

For $L$ measures with equal weights, the selection penalty is
$\Delta_d\sqrt{p+\log L}$. This is useful when several deterministic measures
are adapted to different likely shapes of the random geometry. The mixture
calculation is a corollary, not a separate majorizing-measure theorem.

There is also an initial-information version. A full-support measure may be
$\F_0$-measurable if the innovation assumptions hold conditionally on $\F_0$.
Apply the proof conditionally. This differs from choosing a measure after all
innovations have been observed.

\subsection{Recovery of a prescribed modulus}

Suppose $d(t,u)\le A(\omega)\rho(t,u)$ for a deterministic pseudometric $\rho$.
By monotonicity and scaling,
\[
 \G_\mu(d)\le A\G_\mu(\rho),\qquad
 \Delta_d\le A\Delta_\rho.
\]
Thus Theorem~\ref{thm:main} contains the familiar deterministic-geometry upper
bound. Its advantage is that such a domination is optional: it may be wasteful
when different outcomes have different well-populated small balls.

\begin{example}[A random two-block geometry]\label{ex:split}
Let $T=\{1,\ldots,m\}$ with even $m$, and let $A$ be a uniformly chosen subset
of size $m/2$, observed initially. Let $g$ be an independent standard Gaussian,
and put $f(t)=\ind_A(t)g$. The square-function metric is
$d(t,u)=|\ind_A(t)-\ind_A(u)|$. Under uniform $\mu$, every ball of radius less
than $1$ has mass $1/2$. Therefore
\[
 \Delta_d=1,\qquad \G_\mu(d)=\sqrt{\log(2e)},\qquad
 \osc_T f=|g|.
\]
The geometry remains inexpensive even though the partition itself is random.

For comparison, suppose $d\le K(\omega)\rho$ for one deterministic
pseudometric $\rho$. Every distinct pair can be separated by $A$, so
$a:=\min_{t\ne u}\rho(t,u)>0$ whenever $K$ is finite almost surely. A pair
attaining this minimum is separated with probability at least $1/2$;
therefore $\E K\ge1/(2a)$. For every fixed probability measure $\nu$, some
point has mass at most $1/m$, and its $\rho$-balls below radius $a$ are
singletons. Thus
\[
 \G_\nu(\rho)\ge a\sqrt{\log(em)},\qquad
 \E K\,\G_\nu(\rho)\ge\tfrac12\sqrt{\log(em)}.
\]
This separates the functionals. It is not by itself an example beyond all
classical methods, since conditioning on the initial partition solves it
immediately.
\end{example}

\section{Haar measure and homogeneous square functions}\label{sec:haar}

Let $T$ now be a finite abelian group, with normalized counting measure $\mu$.
Suppose the terminal metric is translation invariant:
\[
 d(t,u)=\sigma(t-u).
\]
Then $\sigma(0)=0$, $\sigma(-h)=\sigma(h)$, and
$\sigma(h+k)\le\sigma(h)+\sigma(k)$. All balls of the same radius have the same
Haar mass, so
\begin{equation}\label{eq:J}
 \G_\mu(d)=\J(\sigma):=
 \int_0^{\Delta_\sigma}
 \sqrt{\log\frac{e}{\mu\{h:\sigma(h)\le r\}}}\,dr,
 \qquad \Delta_\sigma=\max_h\sigma(h).
\end{equation}
Theorem~\ref{thm:main} immediately gives
\begin{equation}\label{eq:HaarLp}
 \norm{\max_j\osc_T f_j}_{L^p}
 \le CD\norm{\J(\sigma)+\sqrt p\,\Delta_\sigma}_{L^p}.
\end{equation}

This is the ball-mass geometry appearing in the homogeneous calculation of
\cite[Section~4]{KMR06}, with Haar measure in place of interval Lebesgue
measure. In that source, an independent Gaussian kernel provides a conditional
Gaussian metric. Here the preceding martingale argument provides the
probabilistic estimate directly.

\subsection{Averaging the homogeneous entropy}

\begin{proposition}[Finite-group entropy averaging]\label{prop:Haaraverage}
If $\sigma:T\to[0,\infty)$ is a random symmetric subadditive function, vanishing at
zero, and $\E\sigma(h)<\infty$ for all $h$, then
\begin{equation}\label{eq:Haaraverage}
 \E\J(\sigma)\le2\J(\bar\sigma),\qquad
 \bar\sigma(h)=\E\sigma(h).
\end{equation}
\end{proposition}

\begin{proof}
Write $M=|T|$ and order the values of $\sigma$ as
$0=a_1\le a_2\le\cdots\le a_M=\Delta_\sigma$, assigning the identity to the
first position. Let $\phi(v)=\sqrt{\log(e/v)}$ on $(0,1]$. Summation by parts
gives
\begin{equation}\label{eq:entropy-sorted}
 \J(\sigma)=\Delta_\sigma+\mathcal A(\sigma),\qquad
 \mathcal A(\sigma)=\sum_{k=2}^M
 \left[\phi\left(\frac{k-1}{M}\right)-
       \phi\left(\frac{k}{M}\right)\right]a_k.
\end{equation}
The function $\phi$ is decreasing and convex. The bracketed coefficients are
therefore nonnegative and decreasing in $k$. Pairing these weights with
increasingly sorted values is the minimum over all pairings. Thus
$\mathcal A$ is a minimum of linear functions of the values
$(\sigma(h))_{h\ne0}$, and hence is concave. Jensen's inequality gives
\[
 \E\mathcal A(\sigma)\le\mathcal A(\bar\sigma).
\]
Subadditivity supplies the remaining estimate. For every $h,k$,
$\sigma(h)\le\sigma(k)+\sigma(h-k)$. Averaging in $k$ and using Haar
invariance yields
\[
 \Delta_\sigma\le2\int_T\sigma\,d\mu,
 \qquad \E\Delta_\sigma\le2\Delta_{\bar\sigma}.
\]
Combine this with \eqref{eq:entropy-sorted}.
\end{proof}

\begin{corollary}\label{cor:HaarL1}
For a family in Theorem~\ref{thm:main} whose square-function metric is
translation invariant,
\begin{equation}\label{eq:HaarL1}
 \E\max_j\osc_T f_j\le CD\J(\E\sigma).
\end{equation}
\end{corollary}

The averaging principle should be viewed alongside the Fernique-based
averaging step in \cite[after (4.13)]{KMR06}. Proposition~\ref{prop:Haaraverage}
is an elementary proof for the finite-group formulation used here. It is not a
claim that arbitrary nonhomogeneous random metrics satisfy
$\E\G_\mu(d)\le C\G_\mu(\E d)$.

If $d$ is not invariant, one may set
\[
 \sigma(h)=\max_{t\in T}d(t+h,t).
\]
This is symmetric and subadditive and gives an invariant majorant of $d$.
The majorization is optional and can be expensive; the fixed-measure theorem
can instead be used on the original metric.

\subsection{An adaptive Fourier example}

\begin{example}[Predictable frequencies and amplitudes]\label{ex:Fourier}
On $T=\mathbb Z/m\mathbb Z$, let $b_i\in\R$ and $k_i\in\mathbb Z/m\mathbb Z$
be predictable. Let $(\gamma_i,\gamma_i')$ be independent standard Gaussian
pairs, independent of the preceding information, and set
\begin{equation}\label{eq:Fourier}
 f_j(t)=\sum_{i=1}^j b_i\left[
 \gamma_i\cos\frac{2\pi k_it}{m}+
 \gamma_i'\sin\frac{2\pi k_it}{m}\right].
\end{equation}
The coefficients and frequencies may depend on all previous Gaussian pairs.
The terminal field is therefore not generally Gaussian. Nevertheless, the
trigonometric identity for the sum of squared sine and cosine differences gives
\begin{equation}\label{eq:Fouriermetric}
 d(t,u)^2=\sigma(t-u)^2
 =2\sum_{i=1}^N b_i^2
       \left[1-\cos\frac{2\pi k_i(t-u)}m\right].
\end{equation}
The metric is exactly translation invariant for every outcome, so
Corollary~\ref{cor:HaarL1} applies without taking a spatial supremum.
If the second moments below are finite, the deterministic modulus
\[
 \tau(h)=\left(2\E\sum_{i=1}^N b_i^2
               \left[1-\cos\frac{2\pi k_ih}{m}\right]\right)^{1/2}
\]
satisfies $\E\sigma(h)\le\tau(h)$. Hence
$\E\max_j\osc_T f_j\le CD\J(\tau)$.
This gives a concrete setting in which the geometry evolves predictably but
remains homogeneous.
\end{example}

\section{Contractions and the stochastic-convolution target}\label{sec:contractions}

\subsection{A common predictable contraction}

Consider
\begin{equation}\label{eq:contraction}
 f_0(t)=0,\qquad
 f_j(t)=V_jf_{j-1}(t)+A_j(t)\gamma_j,
\end{equation}
where $V_j$ is a predictable linear contraction on $X$, the same for all
$t\in T$. Define $d$ from the innovations $A_i(t)\gamma_i$ exactly as in
\eqref{eq:metric}.

\begin{proposition}[Common-contraction version]\label{prop:contraction}
The conclusions of Theorem~\ref{thm:main} remain valid for
\eqref{eq:contraction} with this innovation metric.
\end{proposition}

\begin{proof}
For a pair $(t,u)$, the difference satisfies the same recursion with innovation
$(A_j(t)-A_j(u))\gamma_j$. Conditional Gaussian comparison and
$\norm{V_jx}\le\norm{x}$ show that \eqref{eq:compensated} is still a
supermartingale.

For interpolation, first apply $V_j$ at the beginning of interval $j$, and then
run its Brownian innovation. The centered field $f-\bar f$ is also contracted,
because a common linear $V_j$ commutes with the fixed $\mu$-average. Its norm
has no upward jump at a contraction. Positive level crossings therefore occur
continuously during an innovation segment.

After a crossing time, the old centered field is carried forward by products
of contractions, so its norm remains at most the crossing level. Subtract this
carried field. The residual starts at zero, satisfies the same subsequent
recursion, and has a remaining innovation metric bounded by $d$. The
conditional localized estimate and the averaging proof apply to it. The
proof of \eqref{eq:goodlambda} is consequently unchanged.
\end{proof}

This proposition does not identify $f$ as a martingale. Nor does it apply to
parameter-dependent $V_j(t)$ by the same difference argument: subtracting two
such recursions produces an additional term involving
$V_j(t)-V_j(u)$.

\subsection{The known convolution inequality as a corollary of the mechanism}

Let $S_k$ be deterministic contractive $C_0$-semigroups on $X$ and let
\[
 \Psi_k(t)=\int_0^tS_k(t-s)\psi_k(s)\,dW_k(s),\qquad
 a_k=\left(\int_0^{T_0}\norm{\psi_k(s)}_{\gamma(H,X)}^2\,ds\right)^{1/2}.
\]
All processes are on a common filtered space; the Brownian noises need not be
independent. We use their continuous versions. The one-process exponential
estimate of van Neerven--Veraar, stated as Theorem~3.3 in the supplied
Cox--van Winden manuscript, is an established input \cite{NV20,CoxVW24}.
Stopping the integrand at an energy level gives
\begin{equation}\label{eq:convtail}
 \Pp\{\Psi_k^*>x,\ a_k\le r\}
 \le3\exp\left(-\frac{x^2}{4D^2r^2}\right),
 \quad \Psi_k^*=\sup_{t\le T_0}\norm{\Psi_k(t)}.
\end{equation}

There is an averaging proof even when the semigroups differ. On
$\{0,1,\ldots,K\}$, add $\Psi_0=0$, $a_0=0$, and use the star pseudometric
\[
 d_*(k,\ell)=a_k+a_\ell\quad(k\ne\ell),\qquad d_*(k,k)=0.
\]
The triangle inequality and a union bound in \eqref{eq:convtail} imply
\begin{equation}\label{eq:startail}
 \Pp\{\sup_t\norm{\Psi_k(t)-\Psi_\ell(t)}>x,
                         \ d_*(k,\ell)\le r\}
 \le6\exp\left(-\frac{x^2}{16D^2r^2}\right).
\end{equation}
The localized averaging proof of Section~\ref{sec:averaging} applies with
changed numerical constants. With the atomic measure \eqref{eq:atomic},
\[
 \G_\mu(d_*)+\sqrt p\,\Delta_{d_*}
 \le C\max_{k\le K}\sqrt{p+\log(k+1)}\,a_k.
\]
At the first crossing of $\max_k\norm{\Psi_k(t)}$ through $\lambda$,
contractivity bounds each carried term
$S_k(t-\tau)\Psi_k(\tau)$ by $\lambda$. The restarted convolutions satisfy
\eqref{eq:convtail} conditionally, and their energies only decrease.
Good-$\lambda$ therefore gives
\begin{equation}\label{eq:CV}
 \left\|\sup_{k\ge1}\sup_{t\le T_0}\norm{\Psi_k(t)}\right\|_{L^p}
 \le CD\left\|\sup_{k\ge1}\sqrt{p+\log(k+1)}\,a_k\right\|_{L^p}.
\end{equation}
Finite families are handled first, and the countable result follows by
monotone convergence. This recovers the form of
\cite[Theorem~3.1]{CoxVW24}, not its optimized constant $10$.

\subsection{How the discrete proof connects to a convolution}

For an elementary predictable integrand, the splitting approximation is
\begin{equation}\label{eq:splitting}
 f_j^\pi=S(t_j-t_{j-1})
       \big(f_{j-1}^\pi+\psi(t_{j-1})(W_{t_j}-W_{t_{j-1}})\big).
\end{equation}
Its innovation energy is at most
\[
 (t_j-t_{j-1})\norm{\psi(t_{j-1})}_{\gamma(H,X)}^2.
\]
Thus the discrete Gaussian-contraction calculation provides the natural
energy bound before passage to continuous time. For bounded finite-rank
step integrands, the splitting approximation converges at every fixed finite
collection of times by strong continuity of $S$ and the usual $L^2$
stochastic-integral estimate. Localized tail bounds pass through these finite
time sets, then through a dense set using the known continuous version.
General integrands are treated by elementary approximation and localization;
\cite{NV22} develops the underlying approximation framework.

This paragraph explains the link with the original discrete plan. It does not
identify the exact one-interval convolution integral with a frozen coefficient:
\[
 \int_{t_{j-1}}^{t_j}S(t_j-s)\psi(s)\,dW(s)
 \quad\hbox{need not equal}\quad
 S(t_j-t_{j-1})\psi(t_{j-1})\Delta W_j.
\]
The latter is a splitting approximation. For arbitrary random evolution
families, adaptedness introduces additional issues, and the general It\^o
formula above is not asserted; see the forward-integral treatment in
\cite[Section~6]{NV22}.

\subsection{The metric for a parameter-dependent kernel}

For
\[
 F_j(t)=\sum_{i\le j}K(t,i)g_i\gamma_i,
\]
where the coefficient maps are predictable for each fixed $t$, the relevant
metric is
\begin{equation}\label{eq:kernelmetric}
 d(t,u)^2=\sum_i\norm{(K(t,i)-K(u,i))g_i}_\gamma^2.
\end{equation}
This is the geometry on which the ball-mass calculation should be performed.
One should not automatically move a spatial supremum inside the sum of
squares. For a deterministic semigroup kernel, the continuous-time energy of
an output increment, $u\le t$, is
\begin{align}
 d(t,u)^2={}&\int_0^u
 \norm{(S(t-r)-S(u-r))g_r}_\gamma^2\,dr\nonumber\\
 &+\int_u^t\norm{S(t-r)g_r}_\gamma^2\,dr.\label{eq:timekernel}
\end{align}
The first term changes the action on old noise; the second adds new noise.
Estimating the masses of balls defined by \eqref{eq:kernelmetric} or
\eqref{eq:timekernel} is a model-specific task, separate from the comparison
theorem.

\section{Moduli of continuity and approximation}\label{sec:continuity}

\subsection{A finite-grid form}

Let $T_L=\{k2^{-L}:0\le k\le2^L\}$ and $0<\alpha<1$. Define
\begin{equation}\label{eq:dyadicS}
 \mathcal S_{\alpha,p}
 =\max_{\substack{0\le r\le L\\1\le k\le2^r}}
 2^{\alpha r}\sqrt{p+r+1}\,
 d(k2^{-r},(k-1)2^{-r}).
\end{equation}
Apply Corollary~\ref{cor:weighted} to the martingales associated with adjacent
dyadic pairs, ordered from coarse to fine, with weights $2^{\alpha r}$.
A level-$r$ pair has an index whose logarithm is comparable to $r+1$.
The deterministic dyadic chaining inequality then gives
\begin{equation}\label{eq:Holder}
 \left\|\max_{j\le N}
  \sup_{t\ne u\in T_L}
       \frac{\norm{f_j(t)-f_j(u)}}{|t-u|^\alpha}\right\|_{L^p}
 \le C_\alpha D\norm{\mathcal S_{\alpha,p}}_{L^p}.
\end{equation}
The constant does not depend on grid depth. The deterministic equivalence
behind this step is also used in \cite[(1.7)]{CoxVW24}.

For example, if
\[
 d(t,u)\le A(\omega)
       \frac{|t-u|^\alpha}{\sqrt{p+\log(e/|t-u|)}},
\]
then the $L^p$ norm of the output H\"older seminorm is bounded by
$C_\alpha D\norm{A}_{L^p}$. If instead $d(t,u)\le A|t-u|^\alpha$, the same
argument gives the enlarged output modulus
\[
 |t-u|^\alpha\sqrt{p+\log(e/|t-u|)}.
\]
Indeed,
$\sum_{\ell\ge r}2^{-\alpha\ell}\sqrt{p+\ell+1}
\le C_\alpha2^{-\alpha r}\sqrt{p+r+1}$.
These are sufficient regularity bounds, not claims of exact moduli for every
family.

\subsection{Why the compact-parameter averaging argument is legitimate}

For a compact metric parameter space $T$, suppose each $t\mapsto A_i(t)$ is
continuous in the Gaussian-radonifying norm, almost surely, and $\mu$ is a
fixed full-support Borel probability measure. The square-function metric is
continuous. For the finite-rank Gaussian sums,
\begin{equation}\label{eq:pathwiseLipschitz}
 \max_j\norm{f_j(t)-f_j(u)}
 \le\left(\sum_i\norm{\gamma_i}_{\R^{m_i}}^2\right)^{1/2}d(t,u).
\end{equation}
We used $\norm{A}_{\mathrm{op}}\le\norm{A}_\gamma$ and Cauchy--Schwarz.
For the continuous interpolation, replace $\norm{\gamma_i}$ by
$\sup_{0\le s\le1}\norm{W_i(s)}$.

Thus averages over $d$-balls of radius $r$ approach the field uniformly as
$r\downarrow0$. Such balls have positive $\mu$-mass because $d$ is continuous
and $\mu$ has full support. Joint measurability follows from the continuous
parameter dependence. The proof of Sections~\ref{sec:averaging} and
\ref{sec:goodlambda} consequently works with integrals over $T$. If needed,
measurability of the ball-mass functional follows from its lower
semicontinuity as a function of $d\in C(T\times T)$. Indeed, under uniform
convergence of metrics, ball masses at a fixed center converge at every radius
except the countably many atoms of the corresponding distance distribution.
Fatou's lemma, followed by the supremum over centers, gives the claimed lower
semicontinuity. A dense countable parameter set determines the continuous
field supremum.

\begin{corollary}[A continuity criterion through approximations]\label{cor:continuity}
Let $F^{(n)}$ be continuous finite-rank fields on a compact metric space, such
that each difference has a representation covered by the preceding compact
version. Denote its square-function metric by $d^{n,m}$. If, for a fixed
full-support $\mu$ and some $p\ge1$,
\begin{equation}\label{eq:criterion}
 \norm{\G_\mu(d^{n,m})+\sqrt p\,\Delta_{d^{n,m}}}_{L^p}
       \longrightarrow0
\end{equation}
and $F^{(n)}(t_0)$ is Cauchy in $L^p$ at one parameter $t_0$, then
$F^{(n)}$ is Cauchy in $L^p(\Omega;C(T;X))$. Its limit has continuous paths.
\end{corollary}

\begin{proof}
The comparison bounds the oscillation of $F^{(n)}-F^{(m)}$. Its supremum
norm is bounded by that oscillation plus its value at $t_0$. Completeness of
$L^p(\Omega;C(T;X))$ proves the assertion.
\end{proof}

Identification with a pre-existing pointwise stochastic integral requires
pointwise convergence of the approximations. The corollary does not silently
assume this identification. In a homogeneous approximation scheme, the finite
Haar argument suggests estimating $\J(\E\sigma^{n,m})$; passage to a compact
group requires the corresponding entropy-averaging limit to be justified, not
merely the existence of Haar measure.

\section{Extensions that the proof supports}\label{sec:extensions}

\subsection{Predictable Rademacher sums}

Consider
\[
 f_j(t)=\sum_{i\le j}a_i(t)\varepsilon_i,
\]
where each $\varepsilon_i$ is a symmetric sign independent of $\F_{i-1}$
and $a_i(t)\in X$ is predictable. Its square-function metric is
\[
 d(t,u)^2=\sum_i\norm{a_i(t)-a_i(u)}^2.
\]
The same fixed-measure estimate holds. Here are the two changes in the proof.

First, the smoothness calculation of \cite[(2.6)]{NV22} implies, for
$g(r)=\frac12[\cosh(\lambda\norm{x+ra})+
\cosh(\lambda\norm{x-ra})]$,
\[
 g''(r)\le D^2\lambda^2\norm{a}^2g(r),\qquad g'(0)=0.
\]
Comparison with the solution of the constant-coefficient differential equation
gives
\begin{align}
 \tfrac12[\cosh(\lambda\norm{x+a})+
           \cosh(\lambda\norm{x-a})]
 &\le\cosh(\lambda\norm{x})\cosh(D\lambda\norm a)\nonumber\\
 &\le\cosh(\lambda\norm{x})
       e^{D^2\lambda^2\norm a^2/2}.\label{eq:Rademacherstep}
\end{align}
This yields the compensated supermartingale and the localized tail.

Second, there is a direct overshoot bound. At a step $i$,
\[
 \sup_t\norm{\Delta(f_i(t)-\bar f_i)}
 \le\sup_t\int_T\norm{a_i(t)-a_i(u)}\,d\mu(u)
 \le\Delta_d.
\]
At a first discrete crossing of $\lambda$, the centered norm is at most
$\lambda+\Delta_d$. On the good-$\lambda$ event with
$D(\G_\mu(d)+\sqrt p\Delta_d)\le\delta\lambda$, this is at most
$(1+\delta)\lambda$. A subsequent crossing of $2\lambda$ requires a
residual exceeding $(1-\delta)\lambda$. The remaining metric decreases as
before, and the conditional localized bound gives absorption after reducing
$\delta$. This proves the analogue of \eqref{eq:main}.

The conclusion is an upper bound. It is not a two-sided characterization of
Bernoulli suprema. Its relation to \cite{BlockDaganRakhlin21} is especially
close and should be addressed before a novelty claim.

\subsection{Finite families of continuous local martingales}

Let $(M_s(t))_{s\le T_0}$, $t\in T$, be a finite family of real-valued
continuous local martingales, all starting at zero and adapted to one
filtration. Set
\begin{equation}\label{eq:bracketmetric}
 d(t,u)=\langle M(t)-M(u)\rangle_{T_0}^{1/2}.
\end{equation}
The matrix of terminal quadratic covariations is positive semidefinite:
quadratic variation of every rational linear combination is nonnegative,
and continuity in the coefficients gives the assertion for all combinations.
Consequently, \eqref{eq:bracketmetric} is a pseudometric.

For $L=M(t)-M(u)$, the exponential local martingale
$\exp(\lambda L_s-\lambda^2\langle L\rangle_s/2)$, with localization,
gives the joint bound
\[
 \Pp\{\sup_{s\le T_0}|L_s|>a,\ \langle L\rangle_{T_0}\le r^2\}
 \le2e^{-a^2/(2r^2)}.
\]
The same estimate holds conditionally after a stopping time, and the remaining
bracket metric only decreases. The averaging and continuous first-crossing
arguments therefore prove
\begin{equation}\label{eq:continuousmart}
 \norm{\sup_{s\le T_0}\osc_T M_s}_{L^p}
 \le C\norm{\G_\mu(d)+\sqrt p\,\Delta_d}_{L^p}.
\end{equation}
Localizing the finitely many quadratic variations and using Fatou's lemma
removes moment restrictions used during the proof. In particular, if
$M_s(t)=\int_0^s h_r(t)\,dL_r$ for one continuous local martingale $L$, then
\[
 d(t,u)^2=\int_0^{T_0}|h_r(t)-h_r(u)|^2\,d\langle L\rangle_r.
\]
The random clock remains part of the metric. This formulation should be read
alongside the earlier continuous-martingale entropy work
\cite{Nishiyama00,VdVvZ05}, not as a first use of brackets in entropy theory.

\subsection{What fails for unrestricted jumps}

A pure predictable-square-function bound cannot hold for arbitrary
martingale differences. Let
\[
 \Pp\{\xi=q^{-1/2}\}=\Pp\{\xi=-q^{-1/2}\}=q/2,
 \qquad \Pp\{\xi=0\}=1-q.
\]
Then $\E\xi=0$ and $\E\xi^2=1$, while
$\norm\xi_{L^p}=q^{1/p-1/2}\to\infty$ for fixed $p>2$ as $q\downarrow0$.
For a two-point parameter set this already prevents an upper bound involving
only the predictable variance metric. A second quantity controlling jumps is
necessary. The discrete martingale theorem in \cite[Theorem~3.6]{CoxVW24}
accordingly contains both a maximal-jump and a square-function term.

One possible research direction is a two-metric random-geometry estimate,
with a Gaussian entropy integral for the quadratic variation and an
exponential entropy integral for the jump scale. Such a theorem is not proved
here. Its joint-ball construction and stopping-time overshoots must be treated
explicitly.

\section{What is established, and what remains to compare}\label{sec:status}

The finite-index Gaussian comparison, its common-contraction version, the
atomic and mixture corollaries, the finite homogeneous entropy average, and
the Rademacher and scalar continuous-martingale versions have proofs in this
manuscript. The convolution bound uses a cited one-process exponential
estimate as its input. Compact-parameter conclusions are stated with
continuity, measurability and approximation assumptions rather than as an
unrestricted theorem for arbitrary indexed families.

The historical literature already contains deterministic majorizing-measure
bounds, quadratic martingale moduli, ball-averaging proofs, and path-dependent
sequential balls. Therefore the appropriate comparison is the exact
self-scaled estimate \eqref{eq:main}: where the metric is random, where the
norm is taken, which measure is fixed, and how the terminal diameter is
handled. In particular, one should check whether a localization of the
intermediate estimates in \cite[Appendix~A]{BlockDaganRakhlin21} already yields
some of the scalar versions stated here. No assertion of priority is made.

For applications, the next useful task is to estimate
$\mu(B_d(t,r))$ for a concrete parameter-dependent kernel. A successful example
should show why keeping the realized geometry improves an actual bound, rather
than merely rewriting a known scalar envelope estimate. The adaptive Fourier
field in Example~\ref{ex:Fourier} is a tractable test case; nonlinear or
nonhomogeneous kernels present a further problem.

\roadmap{\textbf{The conclusion of the argument.}
The square function is not only an intermediate variance estimate. Applied to
parameter differences, it is the metric on which the averaging argument is
performed. The fixed measure pays for selecting its random balls, and
stopping-time localization pays for using the random diameter. This is the
precise sense in which the original comparison plan leads to entropy and
continuity estimates.}

\appendix
\
\section{Why the averaging measure cannot be optimized for free}\label{app:adaptive}

We give a concrete obstruction. Let $\gamma_1,\ldots,\gamma_n$ be independent
standard Gaussians, let $b_i=\operatorname{sgn}\gamma_i$, and index the family by
$\theta\in\{-1,1\}^n$. Put
\[
 a_i^\theta=\theta_i\prod_{k<i}\ind_{\{\theta_k=b_k\}},\qquad
 X_j^\theta=\sum_{i\le j}a_i^\theta\gamma_i.
\]
The coefficients are predictable. For the realized branch $\theta=b$,
$X_n^b=\sum_i|\gamma_i|$. Every other branch agrees until its first mismatch,
takes one negative increment, and stops. Therefore
\begin{equation}\label{eq:tree-large}
 \sup_\theta|X_n^\theta|=\sum_{i=1}^n|\gamma_i|,
 \qquad \E\sup_\theta|X_n^\theta|=\sqrt{2/\pi}\,n.
\end{equation}
The oscillation is at least this large: a branch disagreeing at the first
coordinate has value $-|\gamma_1|$.

For fixed realized $b$, however, the coefficient vectors fall into only $n+1$
classes: first mismatch at $k=1,\ldots,n$, or no mismatch. For two mismatch
classes $k<\ell$, their squared Euclidean distance is $4+\ell-k$; the full
branch has squared distance $4+n-k$ from class $k$. Hence the terminal
square-function metric has diameter at most $\sqrt{n+3}$ and is comparable,
up to an additive constant inside the square, to distance on a line followed
by a square root.

Choose a terminal random full-support measure giving mass $1/(n+1)$ to each
class and distributing that mass uniformly within the class. For $r<3$, every
ball has mass at least $1/(n+1)$. For $3\le r\le\sqrt{n+3}$ it contains at
least $c\min(r^2,n+1)$ consecutive classes, with a numerical $c>0$. Thus
\[
 \G_{\mu_\omega}(d_\omega)
 \le C\sqrt{\log(n+1)}+
 C\int_0^{\sqrt n}\sqrt{\log\frac{en}{r^2}}\,dr
 \le C'\sqrt n.
\]
The last integral is evaluated by $r=\sqrt n\,x$ and the integrability of
$\sqrt{\log(e/x^2)}$ on $(0,1)$.

Together with \eqref{eq:tree-large}, this rules out a universal bound with
$\E\inf_\mu\G_\mu(d_\omega)$ in place of the fixed-measure right-hand side.
The counterexample does not contradict Theorem~\ref{thm:main}: the convenient
measure here was chosen from the complete realized branch. Corollary~\ref{cor:mixture}
permits a prescribed collection of measures only with its explicit selection
penalty.

\end{document}